\documentclass[12pt,reqno]{article}

\usepackage[usenames]{color}
\usepackage{amssymb}
\usepackage{amsmath}
\usepackage{amsthm}
\usepackage{amsfonts}
\usepackage{amscd}
\usepackage{graphicx}
\usepackage{xcolor}

\usepackage[colorlinks=true,
linkcolor=webgreen,
filecolor=webbrown,
citecolor=webgreen]{hyperref}

\definecolor{webgreen}{rgb}{0,.5,0}
\definecolor{webbrown}{rgb}{.6,0,0}

\usepackage{color}
\usepackage{fullpage}
\usepackage{float}

\usepackage{graphics}
\usepackage{latexsym}
\usepackage{epsf}
\usepackage{breakurl}
\usepackage{fullpage}

\newcommand{\seqnum}[1]{\href{https://oeis.org/#1}{\rm \underline{#1}}}

\DeclareMathOperator{\per}{per}
\DeclareMathOperator{\cexp}{ce}

\DeclareMathOperator{\CAL}{CAL}
\DeclareMathOperator{\CAN}{CAN}

\def\Enn{\mathbb{N}}

\makeatletter
\let\@fnsymbol\@arabic
\makeatother

\begin{document}

\theoremstyle{plain}
\newtheorem{theorem}{Theorem}
\newtheorem{corollary}[theorem]{Corollary}
\newtheorem{lemma}[theorem]{Lemma}
\newtheorem{proposition}[theorem]{Proposition}

\theoremstyle{definition}
\newtheorem{definition}[theorem]{Definition}
\newtheorem{example}[theorem]{Example}
\newtheorem{conjecture}[theorem]{Conjecture}
\newtheorem{problem}[theorem]{Problem}

\theoremstyle{remark}
\newtheorem{remark}[theorem]{Remark}

\newtheorem{observation}[theorem]{Observation}
\title{Complement Avoidance in Binary Words}

\author{
James Currie\footnote{Department of Math/Stats, University of Winnipeg, 515 Portage Ave., Winnipeg, MB, R3B 2E9, Canada; e-mail
\href{mailto:j.currie@uwinnipeg.ca}{\tt j.currie@uwinnipeg.ca}.}
\and
L{\!'}ubom{\'i}ra Dvo\v{r}\'akov\'a\footnote{FNSPE Czech Technical University, 
Prague, Czech Republic;\newline e-mail \href{mailto:lubomira.dvorakova@fjfi.cvut.cz}{\tt lubomira.dvorakova@fjfi.cvut.cz}.}
\and
Pascal Ochem\footnote{LIRMM, CNRS, Universit\'e de Montpellier, France; e-mail \href{ochem@lirmm.fr}{\tt ochem@lirmm.fr}.}
\and
Daniela Opo\v{c}ensk\'a\footnote{FNSPE Czech Technical University, 
Prague, Czech Republic; e-mail
\href{mailto:opocedan@fjfi.cvut.cz}{\tt opocedan@fjfi.cvut.cz}.}
\and
Narad Rampersad\footnote{Department of Math/Stats, University of Winnipeg, 515 Portage Ave., Winnipeg, MB, R3B 2E9, Canada; e-mail
\href{mailto:n.rampersad@uwinnipeg.ca}{\tt n.rampersad@uwinnipeg.ca}.} \and
Jeffrey Shallit\footnote{School of Computer Science,
University of Waterloo,
Waterloo, ON  N2L 3G1,
Canada; e-mail
\href{mailto:shallit@uwaterloo.ca}{\tt shallit@uwaterloo.ca}.}}

\maketitle

\begin{abstract}
The complement $\overline{x}$ of a binary word $x$ is obtained by changing each $0$ in $x$ to $1$ and vice versa.
We study infinite binary words $\bf w$ that avoid sufficiently large complementary factors; that is, if $x$ is a factor of $\bf w$, then $\overline{x}$ is not a factor of $\bf w$.  In particular, we classify such words according to their critical exponents. 
\end{abstract}

\section{Introduction}

Let $x$ be a finite nonempty binary word.  We write $\overline{x}$
for the complementary word, image of the
 morphism that maps
$0 \rightarrow 1$ and $1 \rightarrow 0$, and we write
$x^R$ for the reversal (mirror image) of $x$.
We say $y$ is a {\it factor\/} of a
(one-sided) infinite word $\bf w$
if ${\bf w} = xy{\bf z}$ for a finite
word $x$ and an infinite word $\bf z$. 
In this paper, we are interested in the construction of and properties of infinite binary words $\bf w$ {\it avoiding\/} complementary factors:  that is, if $x$ is a nonempty factor of $\bf w$, then $\overline{x}$ is not.   This is not a new notion; for example, complement avoidance in de Bruijn words was studied by Sawada et al.~\cite{Sawada&Stevens&Williams:2011}.

Evidently it is impossible for an infinite word to avoid complementary factors of {\it all\/} lengths, except in the trivial cases 
$0^\omega = 000\cdots$ and
$1^\omega = 111\cdots$.  A natural question then poses itself: are there such infinite words if the set of exceptions is restricted in some way, say  by length or by cardinality?   And what is the repetition
threshold of such infinite words?
We now turn to repetitions.

We
say that a finite word $w=w[1..n]$ has {\it period\/}
$p\geq 1$ if $w[i]=w[i+p]$ for $1 \leq i \leq n-p$.
The smallest period of a word $w$ is called
{\it the\/} period, and we write it as
$\per(w)$.  The {\it exponent\/} of a finite word
$w$, written $\exp(w)$ is defined to be
$|w|/\per(w)$.   For a real number
$\alpha$, we say a word
(finite or infinite) is
{\it $\alpha$-free\/} if the exponent of all its nonempty factors is $<\alpha$.    We say a word is {\it $\alpha^+$-free\/} if the exponent of all its nonempty factors is $\leq\alpha$.  A word that is $2$-free is also called squarefree, and a word that is $3$-free is also called cubefree.
A word that is $2^+$-free is also called overlap-free.

The {\it critical exponent\/} of
a finite or infinite word $x$ is the supremum, over all
nonempty finite factors $w$ of $x$, of
$\exp(w)$; it is written $\cexp(x)$.  The {\it repetition
threshold} for a language $L$ of infinite words is
defined to be the infimum, over all
$x \in L$, of $\cexp(x)$.

The repetition thresholds for various classes of words have been studied extensively.  To name just a few classes, Dejean \cite{Dejean:1972} determined the repetition threshold for all words over a $3$-letter alphabet, and conjectured its value
for larger alphabets.  Her conjecture attracted a lot of attention, and was finally resolved by Rao
\cite{Rao:2011} and Currie and Rampersad \cite{Currie&Rampersad:2011}, independently.  

Other classes that have been studied include the Sturmian words, studied by Carpi and de Luca
    \cite[Prop.~15]{Carpi&deLuca:2000};
the palindromes, studied in
    \cite{Shallit:2016};
    the rich words, studied by Currie et al.~\cite{Currie&Mol&Rampersad:2020}; the
    balanced words, studied by
    Rampersad et al.~\cite{Rampersad&Shallit&Vandomme:2019} and
    Dvo\v{r}\'akov\'a et al.~\cite{Dvorakova&Opocenska&Pelantova&Shur:2022}; and the
    complementary symmetric
    Rote words, studied by
    Dvo\v{r}\'akov\'a et al.~\cite{Dvorakova&Medkova&Pelantova:2020}.
Other related works include
\cite{Ilie&Ochem&Shallit:2005,Badkobeh&Crochemore:2011,Fiorenzi&Ochem&Vaslet:2011,Samsonov&Shur:2012,Mousavi&Shallit:2013}.

In this paper we study the repetition threshold for two classes of infinite words:
\begin{itemize}
    \item $\CAL_\ell$, the binary words for which there is no length-$\ell$ word
    $x$ such that both $x$ and
    $\overline{x}$ appear as factors;
    \item $\CAN_n$, the binary words for which there are at most $n$ distinct words
    $x$ such that both $x$ and
    $\overline{x}$ appear as factors.
\end{itemize}
It turns out that there is an interesting and subtle hierarchy, depending on the values of $\ell$ and $n$.

Our work is very similar in flavor to that of
\cite{Shallit:2004}, which found a similar hierarchy concerning critical exponents and sizes of squares avoided.  The hierarchy for complementary factors, as we will see, however, is significantly more complex.

We will need the following famous infinite words.

\begin{itemize}
    \item The Fibonacci word ${\bf f} = 0100101001001010010100100 \cdots$,
    fixed point of the morphism $0 \rightarrow 01$, $1 \rightarrow 0$.
    See, for example, \cite{Berstel:1986b}.

    \item The word ${\bf p} = 0121021010210121010210121 \cdots$,
    fixed point of the morphism $\varphi$ sending $0 \rightarrow 01$, $1 \rightarrow 21$, $2 \rightarrow 0$.   This is sequence \seqnum{A287072} in the OEIS.
    Its properties were recently studied in~\cite{Currie&Ochem&Rampersad&Shallit:2022}.
\end{itemize}

The paper is organized as follows.  In Section~\ref{calsec}, we introduce the class $\CAL_\ell$ mentioned above and we establish the hierarchy alluded to previously.  Section~\ref{cansec} does the same thing for the
class $\CAN_n$.  In both cases we need some critical exponent calculations, which are carried out in Section~\ref{critsec}.   Finally, in Section~\ref{threshsec} we study finite words avoiding complementary factors and determine under what conditions there are exponentially many such words.

\section{The class \texorpdfstring{$\CAL_\ell$}{CAL}}
\label{calsec}

In this section, we investigate the repetition threshold for
the class $\CAL_\ell$, the binary words $\bf w$ with the property
that if $x$ is a length-$\ell$ factor of $\bf w$, then $\overline{x}$ is not.   We will need two additional morphisms:  the Thue-Morse morphism $\mu$, which maps
$0 \rightarrow 01$ and $1 \rightarrow 10$, and
the morphism $\psi$, defined as follows:
\begin{align*}
0 & \rightarrow 011001\\
1 & \rightarrow 0\\
2 & \rightarrow 01101 .
\end{align*}

\begin{lemma}
Suppose $\bf x$ is an infinite binary word avoiding $(7/3)$-powers.
% \todo[inline]{L: shouldn't it be $<\frac{7}{3}$?}
% It should indeed.
Then $\bf x$ contains infinitely many (and hence, arbitrarily large) complementary factors.
\label{7/3}
\end{lemma}

\begin{proof}
By a result of Karhum\"aki and Shallit \cite{Karhumaki&Shallit:2004}, every
infinite binary word avoiding $e$-powers for $e \leq \tfrac73$ contains
$\mu^n (0)$ as a factor for all $n \geq 1$.   Such a word is of the form
$\mu^{n-1}(01) = \mu^{n-1}(0) \mu^{n-1}(1)$, and these two terms are
complementary factors of length $2^{n-1}$.
\end{proof}

By Lemma~\ref{7/3}, the lower limit on the repetition threshold is $\tfrac73$.
% since if a word is $(7/3)$-free, it contains arbitrarily large factors
% of the form $\mu^n (0)$ for $\mu$ the Thue-Morse morphism.  
For $\ell = 1$ the only such words are $0^\omega$ and $1^\omega$.
For $\ell = 2$ the only such words are $0^\omega$, $1^\omega$, $10^\omega$, and $01^\omega$.
% \todo[inline]{L: the notation $\epsilon +0$ has not been introduced}
Larger $\ell$ are handled in Theorem~\ref{complem1} below, but to prove it we first need to provide some terminology and a lemma from \cite{Mol&Rampersad&Shallit:2020}.
A morphism $f:\Sigma^*\rightarrow \Delta^*$ is called \emph{$q$-uniform\/} if $|f(a)|=q$ for all $a\in\Sigma$, and is called \emph{synchronizing\/} if for all $a,b,c\in\Sigma$ and $u,v\in \Delta^*$, if $f(ab)=uf(c)v$, then either $u=\varepsilon$ and $a=c$, or $v=\varepsilon$ and $b=c$.
The following result is quoted almost verbatim  from \cite[Lemma 23]{Mol&Rampersad&Shallit:2020}:
\begin{lemma}
Let $a,b\in\mathbb{R}$ satisfy $1<a<b$.  Let $\alpha\in\{a,a^+\}$ and $\beta\in\{b,b^+\}$.  Let $h\colon \Sigma^*\rightarrow \Delta^*$ be a synchronizing $q$-uniform morphism.  Set
\[
t = \max\left(\frac{2b}{b-a},\frac{2(q-1)(2b-1)}{q(b-1)}\right).
\]
If $h(w)$ is $\beta$-free for every $\alpha$-free word $w$ with
$|w|\leq t$, then $h(z)$ is $\beta$-free for every $\alpha$-free word $z \in \Sigma^*$.
\label{mrs}
\end{lemma}
We will use this lemma as follows:  through an exhaustive search, we find an appropriate uniform morphism from $\{0,1,2\}^* \rightarrow \{0,1\}^*$, and then we apply this morphism to an arbitrary ternary squarefree word.  Then we use the fact that there are uncountably many infinite ternary squarefree words, and exponentially many finite ternary squarefree words \cite{Shur:2012}.

We are now ready to state and prove our result on avoiding complementary factors.

\begin{theorem}
There exists an infinite $\beta^+$-free binary word containing
no complementary factors of length $\geq \ell$, for the following
pairs $(\ell, \beta)$. Moreover, this list of pairs is optimal.
%\todo[inline]{L: shouldn't we write already in the statement that the pairs are optimal, and moreover optimal are also $(6,3), (9,5/2), (10,5/2)$ etc.}
\begin{itemize}
    \item[(a)] $(3,2+\alpha)$, where $\alpha=(1+\sqrt{5})/2$.
% \todo[inline]{L: it is necessary to say that $\alpha$ is the golden mean}
    \item[(b)] $(5,3)$ 
    \item[(c)] $(7, \tfrac83)$
    \item[(d)] $(8, \tfrac52)$
    % \item[(e)] $(11, \gamma)$, where 
    % $\gamma \doteq 2.480862716147236962394265321$ is the real zero of the polynomial 
    % $5X^3-26X^2+43X-23$.
    \item[(e)] $(11, \gamma')$, where
    $\gamma' \doteq 2.4808627161472369$ is the critical exponent of $\bf p$.
    \item[(f)] $(13,\tfrac73)$.
\end{itemize}
\label{complem1}
\end{theorem}

\begin{proof}
\leavevmode
\begin{itemize}
    \item[(a)]
    $(3,2+\alpha)$ is achieved by any Sturmian word $\bf x$ with slope
    $[0,3,1,1,1,1,\ldots] = (5-\sqrt{5})/10$.   As is well-known, a Sturmian word has exactly $n+1$ factors of length $n$, and these factors are independent of the intercept of the Sturmian word.   For intercept $0$, the first $6$ symbols are
    $0 0 1 0 0 0$, and so the four factors of $\bf x$ of length $3$ are $001,010,100,000$, and no complement of these words appears as a factor of $\bf x$.   
    
    On the other hand, we know from 
    the proof of Proposition 15 of
    \cite{Carpi&deLuca:2000} that
    the critical exponent of $\bf x$
    is $2+\alpha$.  Thus, since we can choose the intercept of a Sturmian word to be any real in $[0,1]$, there are uncountably many binary words
    in $\CAL_3$ with critical exponent $2+ \alpha$.
    
    This is best possible, as shown in Theorem~\ref{l=4}.
    
%This is achievable by applying the morphism $\eta$ mapping
%$0 \rightarrow 0$, $1 \rightarrow 01$ to the Fibonacci word $\bf f$.
%(This word %$\eta({\bf f})$ is, by the way, the Sturmian word with slope
%$(5-\sqrt{5})/10$ and intercept $(15-3\sqrt{5})/10$.)

%We can prove that (i) avoids complementary factors of length $\geq
%3$ and (ii) has critical exponent $2+\alpha$ with {\tt Walnut}.
%% zzzz
\item[(b)]
% $(5, 3^+)$ is achieved by the morphic word defined by applying the morphism
% $0 \rightarrow 01$,  $1 \rightarrow 0$ to the Thue-Morse word $\bf t$.
% This word satisfies $t=8$.
% For $l = 5,6$ the smallest critical exponent is $3^+$.   This can be proven in two ways; either use $\zeta_3 ({\bf vtm})$, or use $\bf vtm$ and apply
% the map $0 \rightarrow 0101101$,  $1 \rightarrow 0111011$,  $2 \rightarrow 1010111$ .
% $(5, 3)$ is achieved by applying the 9-uniform morphism
% \begin{align*}
%     0 &\rightarrow 001010001 \\
%     1 &\rightarrow 000101001 \\
%     2 &\rightarrow 000100101,
% \end{align*} 
$(5, 3)$ is achieved by applying the 17-uniform morphism $h_1$ defined by
\begin{align*}
    0 &\rightarrow 01000101000101001 \\
    1 &\rightarrow 01000101000100100 \\
    2 &\rightarrow 01000101000100010,
\end{align*} 
to any ternary squarefree word $\bf w$.  It is easy, by checking all squarefree words of length $5$, to ensure that $h_1({\bf w})$
contains no complementary factors of length $\geq 6$.
To verify the $3^+$-freeness of these words, we use Lemma~\ref{mrs} with
$\alpha = 2$, 
$\beta = 3^+$, $q = 17$, $t = 6$, and
check that the morphism is indeed synchronizing and that the image of every ternary squarefree
word of length $\leq 6$ is $3^+$-free.  This gives uncountably many infinite binary words and exponentially many finite binary words
with the desired avoidance property.

We use this same technique to verify the $\beta^+$-freeness of every word in this paper that is obtained
with a uniform morphism.

To see that this result is optimal, backtracking easily shows that the longest word that contains no complementary factors of length $\geq 6$ and no cubes is of length $50$, and one example is $$00101001001101001001101001101001001101001101001001.$$

\item[(c)]
% $(7, (8/3)^+)$ is achieved by the morphic word defined by applying the morphism
% \begin{align*} 
% 0 &\rightarrow 001100100101001100101001 \\
% 1 &\rightarrow 100100101001001100101001 
% \end{align*}
% to the Thue-Morse word $\bf t$.
$(7, \tfrac83)$ is achieved by applying the 36-uniform morphism
\begin{align*}
    0 &\rightarrow 001001010011001010010011001001010011 \\
    1 &\rightarrow 001001010010011001010011001010010011 \\
    2 &\rightarrow 001001010010011001001010011001010011,
\end{align*} 
to any ternary squarefree word.  This gives uncountably many infinite binary words and exponentially many finite binary words.
% found by Lucas Mol
% message of May 16 2022

The longest word that contains no complementary factors of length $\geq 7$ and no $\tfrac83$-powers
is of 
length $51$ and one example is
$$001001100100110010100110010011001010011001010010100.$$

\item[(d)]
$(8, \tfrac52)$ is achieved by applying the morphism
$\xi$, defined as follows
\begin{align*}
    0 &\rightarrow 01 \\
    1 &\rightarrow 0110 \\
    2 & \rightarrow 1,
\end{align*}
to the infinite word $\bf p$ mentioned above.

The proof that the critical exponent of $\xi({\bf p})$ is $5/2$ is
given in Section~\ref{critsec} starting from Section~\ref{subsect:xi_bispecial factors}.

The longest binary word containing no complementary factors of length $\geq 10$ and no  $\tfrac52$-powers is of length $75$
and one example is\\
$001011010011001011010011011001011010011001011010011011001011010011011001100$.  

\item[(e)]
% $(11, \gamma')$ is achieved by the morphic word defined by applying the
% morphism $\psi$, defined as follows
% \begin{align*}
%     0 & \rightarrow 011001\\
%     1 & \rightarrow 0\\
%     2 & \rightarrow 01101 ,
% \end{align*}
% to the infinite word $\bf p$ of Section~\ref{notation}.

$(11, \gamma')$ is achieved by the word $\psi({\bf p})$.  In Section~\ref{critsec} we show that the critical exponent of $\psi({\bf p})$ is the same as that for $\bf p$.

For $\ell = 12$ and $e=\gamma'$, the optimality is proved as follows.
Let $z=1001011001$. We can check that $\psi({\bf p})$ avoids $z$
and contains $\overline{z}$, $z^R$, and $\overline{z^R}$.
Let $x$ be the prefix of length $40$ of $\psi({\bf p})$.
A computer check shows that there is no infinite $\tfrac52$-free
binary word with $\ell \leq 12$ that avoids simultaneously
$x$, $\overline{x}$, $x^R$, and $\overline{x^R}$.

By symmetry, we consider a bi-infinite $\tfrac52$-free
binary word $\bf w$ with $\ell \leq 12$ that contains~$x$.
Let $X$ be the set containing the complements
of the factors of length $12$ of $x$.
Thus $\ell \leq 12$ means that $\bf w$ avoids $X$.
%So $\bf w$ belongs to the language $L_2$ of $\tfrac52$-free words avoiding $X$.

We compute the set $S$ of factors $f$ such that $efg$ is $\tfrac52$-free and avoids $X$ and $|e| = |f| = |g| = 100$.
We compute the set $S'$ of factors of $\psi({\bf p})$ of length $100$.
We verify that $S = S'$.
This means that ${\bf w}=\psi({\bf v})$ for some bi-infinite ternary word $\bf v$.
Moreover, by considering the pre-images by $\psi$ of ${\bf w}=\psi({\bf v})$
and $\psi({\bf p})$, this implies that $\bf v$ and $\bf p$ have the same set
of factors of length $100/\max(|\psi(0)|,|\psi(1)|,|\psi(2)|)=16$.
In particular, $\bf v$ avoids the set
$\{00, 11, 22, 20, 212, 0101, 02102, 121012, 01021010, 21021012102\}$
mentioned in \cite[Theorem 14]{Currie&Ochem&Rampersad&Shallit:2022}.

Also, $\bf v$ is cube-free since $\bf w$ is $\tfrac52$-free.
By \cite[Theorem 14]{Currie&Ochem&Rampersad&Shallit:2022},
we know that $\bf v$ has the same set of factors as $\bf p$.
Thus $\bf w$  has the same set of factors as $\psi({\bf p})$.
So, the critical exponent of $\psi({\bf p})$ is optimal for $\ell = 12$.
\item[(f)]
% $(13, (7/3)^+)$ is achieved by the morphic word defined by applying the
% morphism 
% \begin{align*}
%     0 & \rightarrow 001100101100110110010110010011001011001101100101101001011001 \\
%     1 & \rightarrow 101100101100100110010110011011001011001001100101101001011001
% \end{align*}
% to the Thue-Morse word $\bf t$.

% Alternatively one can apply the 94-uniform morphism
% \begin{align*}
% 0 & \rightarrow
% 01001100101101001100100110100110010110100110010\\
% & \quad\ \ 
% 11001001100101101001100101100100110100110010110 \\
% 1 & \rightarrow
% 10011001011001001100101101001100100110100110010 \\
% & \quad\ \ 
% 11001001101001100101101001100100110100110010110
% \end{align*}
% to $\bf t$.

%  $(13, (7/3)^+)$ is achieved by applying the morphism
% \begin{align*}
%     0 &\rightarrow 00100110010110100101100110110010011001011001101100101101001011 \\
%     1 &\rightarrow 00100110010110011011001011010010110011011001001100101101001011 \\
%     2 &\rightarrow 00100110010110011011001001100101101001011001101100101101001011,
% \end{align*} 
% to any ternary squarefree word over $\{0,1,2\}$.  This gives uncountably many infinite binary words and exponentially many finite binary words.
% found by Lucas Mol message of May 16 2022
% That is great, but let us rather use a slightly larger morphism
% that is also optimal for CAL_n.

 $(13, \tfrac73)$ is achieved by applying the 69-uniform morphism
\begin{align*}
    0 &\rightarrow 001001100101101001100101100100110100110010110100110010011010011001011 \\
    1 &\rightarrow 001001100101101001100100110100110010110100110010110010011010011001011 \\
    2 &\rightarrow 001001100101101001100100110100110010110010011010011001011010011001011
\end{align*} 
to any ternary squarefree word.  This gives uncountably many infinite binary words and exponentially many finite binary words.

By Proposition~\ref{7/3}, if $e = \tfrac73$, then there are arbitrarily long complemented words.
\end{itemize}

\end{proof}

\section{The class \texorpdfstring{$\CAN_n$}{CAN}}
\label{cansec}

One could also try to minimize the total number $n$ of complemented
words that appear.  Obviously $n$ has to be even.   Recall that
$\CAN_n$ denotes the set of binary words $\bf w$ for which there are at most $n$ distinct words
    $x$ such that both $x$ and
    $\overline{x}$ appear as factors of $\bf w$.
% Once again the smallest possible critical exponent is $(7/3)^+$,
% because for $e \leq 7/3$ every infinite $e$-free binary word
% has infinitely many complemented pairs.
% This appears as a regular item of the theorem.
For $n = 0$ the only such infinite words are $0^\omega$ and $1^\omega$.
For $n = 2$ the only such infinite words are $0^\omega$, $1^\omega$, $01^\omega$, and $10^\omega$.
For larger $n$ the situation is summed up in the following theorem:

\begin{theorem}\label{complem2}
There exists an infinite $\beta^+$-free binary word having at most $n$ complemented words, for the following
pairs $(n,\beta)$. Moreover, this list of pairs is optimal.
\begin{itemize}
    \item[(a)] $(4, 2+\alpha)$
    \item[(b)] $(8,3)$
    \item[(c)] $(24,\tfrac83)$
    \item[(d)] $(36,\tfrac52)$
    \item[(e)] $(64,\gamma')$, where
    $\gamma' \doteq 2.4808627161472369$ is the critical exponent of $\bf p$.
    \item[(f)] $(90, \tfrac73)$.
\end{itemize}
\end{theorem}

\begin{proof}
For the positive part, we use the same words as in Theorem~\ref{complem1}. That is, for every $\beta\in\{2+\alpha, 3, \tfrac83, \tfrac52, \gamma', \tfrac73\}$,
the infinite $\beta^+$-free binary words given in Theorem~\ref{complem1}
to achieve the pair $(\ell,\beta)$ also achieve the pair $(n,\beta)$ in Theorem~\ref{complem2}.
It is not hard to count the complemented factors in such words since their length is less than $\ell$.
Now let us consider the negative part.
\leavevmode
\begin{itemize}
    \item[(a)]
%Take the Fibonacci word $\bf f$ and apply the map $0 \rightarrow 0001$, $1 \rightarrow 001$.
%This can be proved by Walnut, but I %did not do it yet.
% As we saw in the proof of Theorem~\ref{complem1} (a), every Sturmian word $\bf x$ with slope
% $(5-\sqrt{5})/10$ has critical exponent $2+\alpha$.  However, again relying on the proof of Theorem~\ref{complem1} (a), the only nonempty factors of $\bf x$ whose complements also appear in $\bf x$
% are $0$, $1$, $01$, and $10$.
For $4 \leq t \leq 6$ the optimal exponent we can avoid is $2+\alpha$.
To see this note that any such word must avoid having both $x$
and $\overline{x}$ as factors if $|x|\geq 4$, since taking
non-empty prefixes of $x$ and $\overline{x}$ gives at least $8$
complemented factors.  Hence, the smallest exponent that can be avoided
is $2+\alpha$ by Theorem~\ref{l=4}.

\item[(b)] For $8 \leq t \leq 22$ one can avoid $3^+$-powers, and this is optimal. 
%This is achieved by Ochem's word $\zeta_3({\bf vtm})$.
The longest word having at most $22$ complemented words, and no cubes  is of length $50$ and one example is\\
$00101001001101001001101001101001001101001101001001$.

\item[(c)] For $24 \leq t \leq 34$ one can avoid $\tfrac83^+$-powers, and this is optimal.
%and this is achieved by $\zeta_6({\bf vtm})$.
The longest word having at most $34$ complemented words and no $ \tfrac83$-powers is of length $51$ and one example is\\ $001001100100110010100110010011001010011001010010100$.  

\item[(d)] For $36 \leq t \leq 62$ one can avoid $\tfrac52^+$-powers and this is optimal.
The longest word having at most $62$ complemented words and no $ \tfrac52$-powers  is of length $73$ and one example is\\
$0010110100110010110100110110010110100110010110100110110010110100110110010$.

\item[(e)] For $64 \leq t \leq 88$ the optimal exponent that can be avoided is $\gamma'$.
We use the proof of optimality of Theorem~\ref{complem1} (e)
and replace the condition $\ell \leq 12$ by $t \leq 88$.

\item[(f)] For $90 \leq t < \infty$, the optimal  exponent that can be avoided is $\tfrac73^+$.
% This is achieved by the $(7/3)^+$ free word of the previous section,
% namely, apply the the 94-uniform morphism
% \begin{align*}
%     0 & \rightarrow
% 01001100101101001100100110100110010110100110010\\
% & \quad\ \ 
% 11001001100101101001100101100100110100110010110 \\
% 1 & \rightarrow
% 10011001011001001100101101001100100110100110010 \\
% & \quad\ \ 
% 11001001101001100101101001100100110100110010110
% \end{align*}
% to $\bf t$.
%
% Alternatively, use the morphism
% \begin{align*}
%     0 &\rightarrow 001001100101101001100101100100110100110010110100110010011010011001011 \\
%     1 &\rightarrow 001001100101101001100100110100110010110100110010110010011010011001011 \\
%     2 & \rightarrow 001001100101101001100100110100110010110010011010011001011010011001011 ,
% \end{align*} 
% and apply it to any squarefree word.
By Proposition~\ref{7/3}, if $e = \tfrac73$, then there are arbitrarily many complemented words.
\end{itemize}

\end{proof}

\newcommand{\LL}{\mathcal L}

\section{Critical exponent of \texorpdfstring{$\xi(\bf{p})$}{xi(p)} and \texorpdfstring{$\psi({\bf p})$}{psi(p)}}
\label{critsec}
%\todo[inline]{L: we suggest to determine the critical exponent of $\xi({\bf p})$ and $\psi({\bf p})$ at once because the same method works in both cases}
The infinite word $\bf p$ is the fixed point of the morphism $\varphi$, where 
\begin{equation}\label{eq:varphi}
\varphi(0) = 01, \varphi(1) = 21, \varphi(2) = 0.
\end{equation}
Therefore, 
${\bf p} = 01 21 0 21 01 0 21 01 21 01 0 21 01 21 0 21 01 21 01 0 \cdots$

The following characteristics of $\bf p$ are known (see \cite{Currie&Ochem&Rampersad&Shallit:2022}):
\begin{itemize}
    \item The factor complexity (number of distinct length-$n$ factors) of $\bf p$ is $2n + 1$. 
    \item The word $\bf p$ is not closed under reversal because $02$ is a factor of $\bf p$, but $20$ is not.
    \item The word $\bf p$ is uniformly recurrent because the morphism $\varphi$ is primitive.
\end{itemize}

The morphism $\psi$ is defined by
\begin{equation}\label{eq:psi}
\psi(0) = 011001,\ \psi(1) = 0,\ \psi(2) = 01101.
\end{equation}

Therefore,
$\psi({\bf p}) = 011001 0 01101 0  011001 01101 0 \cdots$

The morphism $\xi$ is defined by
\begin{equation}\label{eq:xi}
\xi(0) = 01,\ \xi(1) = 0110,\ \xi(2) = 1.
\end{equation}

Hence,
$\xi({\bf p}) = 01 0110 1 0110 01 1 0110 01 0110 01 1 0110 \cdots$

In this section, we compute the critical exponents of $\psi({\bf p})$ and $\xi({\bf p})$ using the lengths of their bispecial factors and their shortest return words. 

First recall the definitions of bispecial and return word. 
Let $\bf u$ be an infinite word and let ${\mathcal L}(\bf u)$ denote the language of all finite factors of $\bf u$.  Then $w\in {\mathcal L}(\bf u)$ is called {\it left special\/} if $aw, bw \in {\mathcal L}(\bf u)$ for two distinct letters $a,b$. A {\it right special\/} factor is defined analogously. The factor $w$ is called {\it bispecial\/} if it is both left special and right special. A factor $r$ of $\bf u$ is a {\it return word \/} to the factor $w$ if $rw \in {\mathcal L}(\bf u)$ and $rw$ contains $w$ exactly twice -- once as a prefix and once as a suffix.  

For a word $u$ over an ordered $d$-letter alphabet $\Sigma$, we define its \emph{Parikh vector} to be the vector of number of occurrences of each letter in $u$. 
\begin{theorem}[\cite{Dolce&Dvorakova&Pelantova:2023}] \label{Asy:computation}
Let $\mathbf{u}$ be a uniformly recurrent aperiodic sequence. Let $(w_n)$ be a sequence of all bispecial factors ordered by their length.
For every $n\in\Enn$, let $r_n$ be a shortest return word to $w_n$ in $\mathbf{u}$. Then
\begin{equation*}    
\cexp(\mathbf{u}) = 1 + \sup\limits_{n \in \Enn} \left\{ \frac{|w_n|}{|r_n|} \right\}\,.
\end{equation*}
\end{theorem}

\subsection{Bispecial factors in \texorpdfstring{$\bf p$}{p}}
In order to determine bispecial factors in $\psi({\bf p})$ and in $\xi({\bf p})$, we need to explore bispecial factors in $\bf p$. 
First, we will look at the left special factors (LS).
Observing the form of $\varphi$, we can see that every LS has left extensions either $\{0,2\}$, or $\{0,1\}$.
\begin{observation}\label{obs::LSp}
    Let $v \in \LL({\bf p})$, $v \neq \varepsilon$. Then \begin{itemize}
        \item $v$ is LS such that $0v, 2v \in \LL({\bf p})$ if and only if $v$ is a prefix of $1\varphi(w)$, where $w$ is LS such that $0w,1w \in \LL({\bf p})$.        
        \item $v$ is LS such that $0v, 1v \in \LL({\bf p})$ if and only if $v$ is a prefix of $\varphi(w)$, where $w$ is LS such that $0w,2w \in \LL({\bf p})$.
    \end{itemize}
\end{observation}

Second, we will look at the right special factors (RS).
By the definition of $\varphi$, every RS has right extensions either $\{1,2\}$, or $\{0,2\}$. 

\begin{observation}\label{obs::RSp}
    Let $v \in \LL({\bf p})$, $v \neq \varepsilon$. Then
    \begin{itemize}
        \item $v$ is RS such that $v1, v2 \in \LL({\bf p})$ if and only if $v$ is a suffix of $\varphi(w)0$, where $w$ is RS such that $w0,w2 \in \LL({\bf p})$.        
        \item $v$ is RS such that $v0, v2 \in \LL({\bf p})$ if and only if $v$ is a suffix of $\varphi(w)$, where $w$ is RS such that $w1,w2 \in \LL({\bf p})$.
    \end{itemize}
\end{observation}

It follows from the form of left and right special factors that we have at most 4 possible kinds of nonempty bispecial factors in $\bf p$. The following statements are obtained when combining Observations~\ref{obs::LSp} and \ref{obs::RSp}.
\begin{corollary}\label{coro:BS}
    Let  $v \in \LL({\bf p})\setminus\{\varepsilon\}$ be a bispecial factor in $\bf p$.
    \begin{enumerate}
        \item If $0v,2v, v0, v2 \in \LL({\bf p})$, there exists $w$ such that $v = 1\varphi(w)$ and $0w,1w, w1,w2 \in \LL({\bf p})$.
        \item If $0v,1v, v1, v2 \in \LL({\bf p})$, there exists $w$ such that $v = \varphi(w) 0$ and $0w,2w, w0,w2 \in \LL({\bf p})$.
        \item If $0v,2v,v1,v2 \in \LL({\bf p})$, there exists $w$ such that $v =1\varphi(w)0$ and $0w,1w, w0,w2\in \LL({\bf p})$.  
        \item If $0v,1v, v0, v2 \in \LL({\bf p})$, there exists $w$ such that $v = \varphi(w)$ and $0w,2w, w1,w2\in \LL({\bf p})$. 
    \end{enumerate}
\end{corollary}
It follows from Corollary~\ref{coro:BS} that all bispecial factors may be constructed starting from the shortest ones in the following way:
$1$ is the shortest bispecial factor with left extensions $\{0,2\}$ and right extensions $\{0,2\}$. Applying the morphism $\varphi$, we obtain the bispecial factor $\varphi(1)0$ with left extensions $\{0,1\}$ and right extensions $\{1,2\}$. The second application of $\varphi$ gives us the bispecial factor $1\varphi^2(1)\varphi(0)$ with left extensions $\{0,2\}$ and right extensions $\{0,2\}$. This process can be iterated infinitely many times
$$1 \to \varphi(1)0 \to 1\varphi^2(1)\varphi(0) \to \varphi(1)\varphi^3(1)\varphi^2(0)0 \to 1\varphi^2(1)\varphi^4(1)\varphi^3(0)\varphi(0) \to \cdots  $$

Similarly, when starting with the bispecial factor 10 with left extensions $\{0,2\}$ and right extensions $\{1,2\}$, we obtain after application of $\varphi$ the bispecial factor $\varphi(10)$ with left extensions $\{0,1\}$ and right extensions $\{0,2\}$. After the second application of $\varphi$, we have the bispecial factor $1\varphi^2(10)0$ with left extensions $\{0,2\}$ and right extensions $\{1,2\}$. We continue analogously and complete thus the list of all bispecial factors
$$10 \to \varphi(10) \to 1\varphi^2(10) 0 \to \varphi(1)\varphi^3(10)\varphi(0) \to 1\varphi^2(1)\varphi^4(10)\varphi^2(0)0\to \cdots$$

The following statement is an immediate consequence of Corollary~\ref{coro:BS}.

\begin{proposition}\label{prop:BS}
    Let $w \in {\mathcal L}({\bf p})\setminus \{\varepsilon\}$. If $w$ is a bispecial factor, then $w$ has one of the following forms: 
    \begin{itemize}
        \item[$\mathcal{A}$)]\label{BS:A} $$w_A^{(n)} = 1\varphi^2(1)\varphi^4(1)\cdots \varphi^{2n}(1)\varphi^{2n-1}(0)\varphi^{2n-3}(0)\cdots \varphi(0)\quad \quad \text{for $n \geq 1$.}$$  
        If $n = 0$, then we set $w_A^{(0)} = 1$.

        The Parikh vector of $w_A^{(n)}$ is the same as of the factor $1\varphi(012)\varphi^3(012)\cdots \varphi^{2n-1}(012)$.
        
        \item[$\mathcal{B}$)]\label{BS:B} $$w_B^{(n)} = \varphi(1)\varphi^3(1)\cdots \varphi^{2n+1}(1)\varphi^{2n}(0)\varphi^{2n-2}(0)\cdots \varphi^2(0)0 \quad \quad \text{for $n \geq 0$.}$$  

        The Parikh vector of $w_B^{(n)}$ is the same as of the factor $012\varphi^2(012)\varphi^4(012)\cdots \varphi^{2n}(012)$.

        \item[$\mathcal{C}$)]\label{BS:C} $$w_C^{(n)} = 1\varphi^2(1)\varphi^4(1)\cdots \varphi^{2n}(1)\varphi^{2n}(0)\varphi^{2n-2}(0)\cdots \varphi^2(0)0 \quad \quad \text{for $n \geq 0$.}$$ 
        
        The Parikh vector of $w_C^{(n)}$ is the same as of the factor $01\varphi^2(01)\varphi^4(01)\cdots \varphi^{2n}(01)$.

        \item[$\mathcal{D}$)]\label{BS:D} $$w_D^{(n)} = \varphi(1)\varphi^3(1)\cdots \varphi^{2n+1}(1)\varphi^{2n+1}(0)\varphi^{2n-1}(0)\cdots \varphi(0)\quad \quad \text{for $n \geq 0$.}$$         

        The Parikh vector of $w_D^{(n)}$ is the same as of the factor $\varphi(01)\varphi^3(01)\cdots \varphi^{2n+1}(01)$.
    \end{itemize}
\end{proposition}

% \todo[inline]{L: the statements about ordinary bispecial factors and 3 return words to any factor are not necessary for our purpose...}
% \begin{lemma}\label{lem:ordinary}
%   Every bispecial factor of $\bf p$ is ordinary, i.e., if $w$ is a bispecial factor in $\bf p$, then $$\#\{awb \in {\mathcal L}({\bf p})  :  a,b \in \{0,1,2\}\}+1=\#\{aw \in {\mathcal L}({\bf p})  :  a \in \{0,1,2\}\}+\#\{wb \in {\mathcal L}({\bf p})  :  b \in \{0,1,2\}\}.$$ 
% \end{lemma}
% \begin{proof}
% The empty word is ordinary because all factors of length two are $10,01,02,12,21$.

% The possible extensions of bispecial factors summarized in Proposition~\ref{prop:BS} are, for all $n \geq 0$,
% \begin{itemize}
% \item [$\mathcal A)$]  $0w_A^{(n)}2, 2w_A^{(n)}0, 0w_A^{(n)}0$;
% \item [$\mathcal B)$]   $1w_B^{(n)}2, 0w_B^{(n)}1, 1w_B^{(n)}1$;
% \item [$\mathcal C)$]  $2w_C^{(n)}2, 2w_C^{(n)}1, 0w_C^{(n)}2$;
% \item [$\mathcal D)$]   $0w_D^{(n)}0, 0w_D^{(n)}2, 1w_D^{(n)}0$.

% \end{itemize}
% \end{proof}

\subsection{The shortest return words in 
\texorpdfstring{$\bf p$}{p}}

% \begin{corollary}
%     Each factor of $\bf p$ has 3 return words. 
% \end{corollary}
% \begin{proof}
%     This is a consequence of Lemma~\ref{lem:ordinary} and of the following theorem from~\cite{Balkova&Pelantova&Steiner:2008}:
%     let $\bf{u}$ be uniformly recurrent. Then each factor of $\bf{u}$ has exactly 3 return words if and only if $u$'s factor complexity is $2n+1$ and all bispecial factors of $\bf{u}$ are ordinary.
% \end{proof}

Let us derive the form of the shortest return words to all bispecial factors in $\bf p$. 
We will make use of them when solving the same problem for $\psi({\bf p})$ and $\xi({\bf p})$. 
\begin{itemize}
    \item The return words to $\varepsilon$ are 0, 1, 2. 
    \item The return words to 1 are 12, 102, 10.
    \item The return words to $10$ are $10$, $102$, $1012$. The shortest one is $10$ and it is a prefix of all of them. 
    \item The return words to $\varphi(1)0$ are $210 = \varphi(1)0$, $21010$, $2101$. The shortest one is $210$ and it is a prefix of all of them.
\end{itemize}

Using this knowledge and the knowledge of how the bispecial factors can be constructed, we obtain the following observation for the shortest return words:
\begin{itemize}
        \item[$\mathcal{A}$)]\label{RW:A} The shortest return words to $w_A^{(0)}=1$ are $12$ and $10$.
        The shortest return word to $w_A^{(n)}$, $n\geq 1$, has the same Parikh vector as the factor $\varphi^{2n-1}(012)$.
       
        \item[$\mathcal{B}$)]\label{RW:B} The shortest return word to $w_B^{(n)}$, $n\geq 0$, has the same Parikh vector as the factor $\varphi^{2n}(012)$.

        \item[$\mathcal{C}$)]\label{RW:C}
            The shortest return word to $w_C^{(n)}$, $n\geq 0$, has the same Parikh vector as the factor $\varphi^{2n}(01)$.
           
        \item[$\mathcal{D}$)]\label{RW:D} 
        The shortest return word to $w_D^{(n)}$, $n \geq 0$, has the same Parikh vector as the factor $\varphi^{2n+1}(01)$.
        \end{itemize}

\subsection{Bispecial factors in 
\texorpdfstring{$\psi(\bf{p})$}{psi(p)}}
Let us start with some simple observations. 
If a factor $v \in \LL(\psi({\bf p}))$ contains the factor 11001 or 1101, then we are able to write 
$v = x\psi(w)y$ uniquely, where $w \in \LL({\bf p})$, and $x$ (resp., $y$) is a proper suffix (resp., proper prefix) of the image of some letter. 
Moreover, if $v$ is a bispecial factor in $\psi({\bf p})$, then $w$ is a bispecial factor in $\bf p$.

\begin{observation}
    Let  $v \in \LL(\psi({\bf p}))$ be a bispecial factor in $\psi({\bf p})$ such that it contains $11001$ or $1101$.
    Then one of the following items is true.
    \begin{enumerate}
        \item\label{theorem:BSApsi} There exists $w\in \LL({\bf p})$ such that $v = 01\psi(w)0110$ and $0w,2w, w0,w2 \in \LL({\bf p})$.
        \item\label{theorem:BSBpsi} There exists $w\in \LL({\bf p})$ such that $v = \psi(w)0$ and $0w,1w, w1,w2 \in \LL({\bf p})$.
        \item\label{theorem:BSCpsi} There exists $w\in \LL({\bf p})$ such that $v = 01\psi(w)0$ and $0w,2w, w1,w2 \in \LL({\bf p})$.  
        \item\label{theorem:BSDpsi} There exists $w\in \LL({\bf p})$ such that $v = \psi(w)0110$ and $0w,1w, w0,w2 \in \LL({\bf p})$.
    \end{enumerate}    
\end{observation}

\begin{corollary}
Let $v \in {\mathcal L}(\psi({\bf p}))$. If $v$ is a bispecial factor containing $11001$ or $1101$, then $v$ has one of the following forms:
    \begin{itemize}
        \item[$\mathcal{A}$)]\label{BS:Apsi} $$v_A^{(n)} = 01\psi( 1\varphi^2(1)\varphi^4(1)\cdots \varphi^{2n}(1)\varphi^{2n-1}(0)\varphi^{2n-3}(0)\cdots \varphi(0))0110 \quad \quad \text{for $n \geq 1$}.$$  
        The Parikh vector of $v_A^{(n)}$ is the same as that of the factor         
    $$000111\psi(1\varphi(012)\varphi^3(012)\cdots \varphi^{2n-1}(012)).$$
        
        \item[$\mathcal{B}$)]\label{BS:Bpsi} $$v_B^{(n)} = \psi(\varphi(1)\varphi^3(1)\cdots \varphi^{2n+1}(1)\varphi^{2n}(0)\varphi^{2n-2}(0)\cdots \varphi^2(0)0) 0 \quad \quad \text{for $n \geq 0$.}$$  
        The Parikh vector of $v_B^{(n)}$ is the same as that of the factor         
    $$0\psi(012\varphi^2(012)\varphi^4(012)\cdots \varphi^{2n}(012)).$$

        \item[$\mathcal{C}$)]\label{BS:Cpsi} $$v_C^{(n)} = 01\psi(1\varphi^2(1)\varphi^4(1)\cdots \varphi^{2n}(1)\varphi^{2n}(0)\varphi^{2n-2}(0)\cdots \varphi^2(0)0)0 \quad \quad \text{for $n \geq 0$.}$$ 
        The Parikh vector of $v_C^{(n)}$ is the same as of the factor         
        $001\psi(01\varphi^2(01)\varphi^4(01)\cdots \varphi^{2n}(01))$.

        \item[$\mathcal{D}$)]\label{BS:Dpsi} $$v_D^{(n)} = \psi(\varphi(1)\varphi^3(1)\cdots \varphi^{2n+1}(1)\varphi^{2n+1}(0)\varphi^{2n-1}(0)\cdots \varphi(0))0110 \quad \quad \text{for $n \geq 0$.}$$     
        The Parikh vector of $v_D^{(n)}$ is the same as of the factor         
        $0011\psi(\varphi(01)\varphi^3(01)\cdots \varphi^{2n+1}(01))$.
    \end{itemize}
\end{corollary}

\subsection{The shortest return words in \texorpdfstring{$\psi({\bf p})$}{psi(p)}}
Knowing the Parikh vectors of the shortest return words to bispecial factors in $\bf p$ and using the unambiguous reconstruction of $w$ from $\psi(w)$ when $\psi(w)$ contains $11001$ or $1101$, we obtain the following observation for the shortest return words:
    
\begin{itemize}
        \item[$\mathcal{A}$)]\label{RW:Apsi} 
        The shortest return word ${r}^{(n)}_A$ to $v_A^{(n)}$, $n\geq 1$, has the same Parikh vector as the factor $\psi(\varphi^{2n-1}(012))$.
        
        \item[$\mathcal{B}$)]\label{RW:Bpsi}         
        The shortest return word ${r}^{(n)}_B$ to $v_B^{(n)}$, $n\geq 0$, has the same Parikh vector as the factor $\psi(\varphi^{2n}(012))$.

        \item[$\mathcal{C}$)]\label{RW:Cpsi}
        The shortest return word ${r}^{(n)}_C$ to $v_C^{(n)}$, $n\geq 0$, has the same Parikh vector as the factor $\psi(\varphi^{2n}(01))$.

        \item[$\mathcal{D}$)]\label{RW:Dpsi}
       The shortest return word ${r}^{(n)}_D$ to $v_D^{(n)}$, $n\geq 0$, has the same Parikh vector as the factor $\psi(\varphi^{2n+1}(01))$.
    \end{itemize}

\subsection{Critical exponent of 
\texorpdfstring{$\psi({\bf p})$}{psi(p)}}
Having determined the lengths of bispecial factors and of their shortest return words in $\psi({\bf p})$, we can use Theorem~\ref{Asy:computation} to compute the critical exponent of $\psi({\bf p})$:

\begin{align*}
\cexp(\psi({\bf p})) &= 1 + \max\left\{A,B,C,D,E\right\}\\
A &= \sup\left\{\frac{|v_A^{(n)}|}{|{r}_A^{(n)}|} : n\geq 1 \right\} 
= \sup\left\{\frac{|000111\psi(1\varphi(012)\varphi^3(012)\cdots \varphi^{2n-1}(012))|}{|\psi(\varphi^{2n-1}(012))|} : n\geq 1 \right\}\,; \\
B &= \sup\left\{\frac{|v_B^{(n)}|}{|{r}_B^{(n)}|} : n\geq 0 \right\} 
= \sup\left\{\frac{|0\psi(012\varphi^2(012)\varphi^4(012)\cdots \varphi^{2n}(012))|}{|\psi(\varphi^{2n}(012))|} : n\geq 0 \right\}\,; \\
C &= \sup\left\{\frac{|v_C^{(n)}|}{|{r}_C^{(n)}|} : n\geq 0\right\} 
= \sup\left\{\frac{|001\psi(01\varphi^2(01)\varphi^4(01)\cdots \varphi^{2n}(01))|}{|\psi(\varphi^{2n}(01))|} : n\geq 0 \right\}\,; \\
D &= \sup\left\{\frac{|v_D^{(n)}|}{|{r}_D^{(n)}|} : n\geq 0 \right\}
= \sup\left\{\frac{|0011\psi(\varphi(01)\varphi^3(01)\cdots \varphi^{2n+1}(01))|}{|\psi(\varphi^{2n+1}(01))|} : n\geq 0 \right\}\,; \\
E &= \left\{\frac{|w|}{|r|} : w \text{\ bispecial not containing 11001 or 1101, } r \text{ the shortest return word to } w \right\}\,.
\end{align*}

\subsubsection{Computation of \texorpdfstring{$E$}{E}}
\begin{table}[htb]
    \centering
        \setlength{\tabcolsep}{3pt}
\renewcommand{\arraystretch}{1.2}

\begin{tabular}{|c|c|c|c|c|c|c|c|c|}
    \hline
     $w$ & 0 & 1 & 01 & 10 & 101 & 010 & 0110 & 1001 \\ 
 \hline
     $r$ & 0 & 1 & 01 & 10 & 101 & 01011 & 011001 & 100 \\
 \hline      
\end{tabular}\label{table:shortBS}
\caption{Short bispecial factors $w$ and their return words $r$}
\end{table}

As summarized in Table~\ref{table:shortBS}, for each bispecial $w$ that does not contain $11001$ and $1101$ and its shortest return word~$r$, we have
$\frac{|w|}{|r|}\leq \frac{4}{3} \doteq 1.333$, therefore $E = 1+\frac{1}{3}$.

\subsubsection{Computation of \texorpdfstring{$A$ and $B$}{A and B}}
If we consider $a_n := |\psi(\varphi^n(012))|$, then we can see that $a_n$ satisfies the recurrence relation 
$a_{n+1} = 2a_n - a_{n-1} + a_{n-2}$
with initial conditions $a_0 = 12, a_1 = 19$, and $a_2 = 32$.

This recurrence relation may be solved and we obtain
$$a_n = A_1\beta_1^n + B_1 \lambda_1^n + C_1 \lambda_2^n,$$
where
$$\beta_1 \doteq 1.75488,\quad \lambda_1 \doteq 0.12256 + 0.74486 i, \quad \lambda_2 = \overline{\lambda_1}$$
are the roots of the polynomial $t^3 - 2t^2 + t - 1$, and
\begin{align*}
    A_1 &= \frac{12|\lambda_1|^2 - 38\operatorname{Re}(\lambda_1) + 32}{|\beta_1-\lambda_1|^2} \doteq 10.6175\,;\\
    B_1 &= \frac{12\beta_1\lambda_2 - 19 (\beta_1+\lambda_2)+32}{(\beta_1 - \lambda_1)(\lambda_2-\lambda_1)} \doteq 0.6912	-0.1330 i\,;\\
    % |b| \doteq 0.7039
    C_1 &= \overline{B_1}\,.
\end{align*}

Let us first determine
$$A' := \limsup_{n\to+\infty} \frac{|v_A^{(n)}|}{|{r}_A^{(n)}|} = \limsup_{n\to+\infty} \frac{|000111\psi(1\varphi(012)\varphi^3(012)\cdots \varphi^{2n-1}(012))|}{|\psi(\varphi^{2n-1}(012))|}\,.$$
We have
\begin{align*}
    A' &= \limsup_{n\to+\infty} \frac{6 + 1 + \sum_{k=1}^n a_{2k-1}}{a_{2n-1}} \\
    &= \limsup_{n\to+\infty} \frac{7 + A_1\sum_{k=1}^n\beta_1^{2k-1} + B_1 \sum_{k=1}^n \lambda_1^{2k-1} + C_1 \sum_{k=1}^n\lambda_2^{2k-1} }{A_1\beta_1^{2n-1} + B_1 \lambda_1^{2n-1} + C_1 \lambda_2^{2n-1}}\\
    &= \lim_{n\to+\infty} \frac{7 + A_1\beta_1\frac{\beta_1^{2n}-1}{\beta_1^2-1} + B_1\lambda_1\frac{\lambda_1^{2n}-1}{\lambda_1^2-1} + C_1 \lambda_2\frac{\lambda_2^{2n}-1}{\lambda_2^2-1}}{ A_1\beta_1^{2n-1} + B_1 \lambda_1^{2n-1} + C_1 \lambda_2^{2n-1}}= \frac{\beta_1^2}{\beta_1^2-1}\doteq 2.4809 \,.\\
    % &= 0 + \lim_{n\to+\infty} \frac{a\beta_1\frac{\beta_1^{2n}-1}{\beta_1^2-1}}{ a\beta_1^{2n-1} + b \lambda_1^{2n-1} + c \lambda_2^{2n-1}}\,. 
    \end{align*}

Next we will show that $A \leq A'$ and thus $A = A'$. 
We want to show for all $n\geq 1$ that
\begin{align*}
    \frac{7 + A_1\sum_{k=1}^n\beta_1^{2k-1} + B_1 \sum_{k=1}^n \lambda_1^{2k-1} + C_1 \sum_{k=1}^n\lambda_2^{2k-1} }{A_1\beta_1^{2n-1} + B_1 \lambda_1^{2n-1} + C_1 \lambda_2^{2n-1}} 
    \quad &\leq^? \quad  \frac{\beta_1^2}{\beta_1^2 - 1}\\
    (\beta_1^2-1)\left(7+ 2\operatorname{Re}\left(B_1\lambda_1\frac{\lambda_1^{2n}-1}{\lambda_1^2 - 1}\right)\right) + A_1\beta_1^{2n+1} - A_1\beta_1
    \quad &\leq^? \quad 2\beta_1^2 \operatorname{Re}(B_1\lambda_1^{2n-1}) + A_1\beta_1^{2n+1}\\
    (\beta_1^2-1) \left(7 + 2\operatorname{Re}\left(B_1\lambda_1\frac{\lambda_1^{2n}-1}{\lambda_1^2 - 1}\right)\right)
    \quad &\leq^? \quad 2\beta_1^2 \operatorname{Re}(B_1\lambda_1^{2n-1}) + A_1\beta_1\,.\\
\end{align*}

Now, on the one hand, for $n \geq 2$, we have
\begin{align*}
    (\beta_1^2-1) \left(7 + 2\operatorname{Re}\left(B_1\lambda_1\frac{\lambda_1^{2n}-1}{\lambda_1^2 - 1}\right)\right)
    \quad     
    &\leq (\beta_1^2-1) \left(7 + 2|B_1||\lambda_1|\frac{|\lambda_1|^{2n}+1}{|\lambda_1^2 - 1|}\right)\\
    &\leq (\beta_1^2-1)\left(7 + 2 |B_1| |\lambda_1|\frac{|\lambda_1|^4+1}{|\lambda_1^2 - 1|}\right)\,.
    \end{align*}
On the other hand,     \begin{align*}
    2\beta_1^2 \operatorname{Re}(B_1\lambda_1^{2n-1}) + A_1\beta_1
    \quad
    &\geq A_1\beta_1 - 2\beta_1^2|B_1||\lambda_1|^{2n-1} \\
    &\geq A_1\beta_1 - 2\beta_1^2|B_1||\lambda_1|^{3}.
    \intertext{And if we substitute the values, we get}
    (\beta_1^2-1)\left(7 + 2 |B_1| |\lambda_1|\frac{|\lambda_1|^4+1}{|\lambda_1^2 - 1|}\right)
    &\leq A_1\beta_1 - 2\beta_1^2|B_1||\lambda_1|^{3}.
    \end{align*}
For $n = 1$, we get $\frac{7 + a_1}{a_1} = 1 + \frac{7}{19}< A'$.

Therefore $$A = A' = \frac{\beta_1^2}{\beta_1^2-1}.$$

Next, we will use the same procedure for $B$.
\begin{align*}
B' &:= \limsup_{n\to+\infty} \frac{|v_B^{(n)}|}{|{r}_B^{(n)}|} \\
&= \limsup_{n\to+\infty} \frac{|0\psi(012\varphi^2(012)\varphi^4(012)\cdots \varphi^{2n}(012))|}{|\psi(\varphi^{2n}(012))|}\\
    &= \limsup_{n\to+\infty} \frac{1 + \sum_{k=0}^n a_{2k}}{a_{2n}} \\
    &= \limsup_{n\to+\infty} \frac{1+ A_1\sum_{k=0}^n\beta_1^{2k} + B_1 \sum_{k=0}^n \lambda_1^{2k} + C_1 \sum_{k=0}^n\lambda_2^{2k} }{A_1\beta_1^{2n} + B_1 \lambda_1^{2n} + C_1 \lambda_2^{2n}}\\
    & = \frac{\beta_1^2}{\beta_1^2-1}\,.
\end{align*}

Now we will show that $B = B'$, because for all $n \geq 0$ we have
\begin{align*}
    \frac{1+ A_1\sum_{k=0}^n\beta_1^{2k} + B_1 \sum_{k=0}^n \lambda_1^{2k} + C_1 \sum_{k=0}^n\lambda_2^{2k} }{A_1\beta_1^{2n} + B_1 \lambda_1^{2n} + C_1 \lambda_2^{2n}} 
    \quad &\leq^? \quad  \frac{\beta_1^2}{\beta_1^2 - 1}\\
    (\beta_1^2-1)\left(1 + 2\operatorname{Re}\left(B_1\frac{\lambda_1^{2n+2}-1}{\lambda_1^2 - 1}\right)\right) + A_1\beta_1^{2n+2} - A_1 
    \quad &\leq^? \quad 2\beta_1^2 \operatorname{Re}(B_1\lambda_1^{2n}) + A_1\beta_1^{2n+2}\\
    (\beta_1^2-1) \left(1 + 2\operatorname{Re}\left(B_1\frac{\lambda_1^{2n+2}-1}{\lambda_1^2 - 1}\right)\right)
    \quad &\leq^? \quad 2\beta_1^2 \operatorname{Re}(B_1\lambda_1^{2n}) + A_1\,.\\
    \end{align*}

    Now, on the one hand, for all $n\geq 2$, we have
    \begin{align*}
    (\beta_1^2-1) \left(1 + 2\operatorname{Re}\left(B_1\frac{\lambda_1^{2n+2}-1}{\lambda_1^2 - 1}\right)\right )
    &\leq (\beta_1^2-1) \left(1+2\operatorname{Re}\left(\frac{B_1}{1-\lambda_1^2}\right) - 2\operatorname{Re}\left(B_1\frac{\lambda_1^{2n+2}}{1-\lambda_1^2}\right)\right)\\
    &\leq (\beta_1^2-1)\left(1+ 2\operatorname{Re}\left(\frac{B_1}{1-\lambda_1^2}\right) + 2|B_1|\frac{|\lambda_1^{2n+2}|}{|1-\lambda_1^2|} \right)\\
    &\leq (\beta_1^2-1)\left(1 + 2\operatorname{Re}\left(\frac{B_1}{1-\lambda_1^2}\right) + 2|B_1|\frac{|\lambda_1|^{6}}{|1-\lambda_1^2|} \right)\,.
    \end{align*}
   On the other hand, $$
    2\beta_1^2 \operatorname{Re}(B_1\lambda_1^{2n}) + A_1 \geq A_1-2\beta_1^2|B_1||\lambda_1|^4.
    $$
    Substituting the values proves that
    $$
     (\beta_1^2-1)\left(1 + 2\operatorname{Re}\left(\frac{B_1}{1-\lambda_1^2}\right) + 2|B_1|\frac{|\lambda_1|^{6}}{|1-\lambda_1^2|} \right) < A_1-2\beta_1^2|B_1||\lambda_1|^4.
    $$
    For $n = 0$, we get $\frac{1 + a_0}{a_0} = \frac{1+12}{12} < B'$.
    For $n = 1$, we get $\frac{1 + a_0 + a_2}{a_2} = 1+\frac{1+12}{32} < B'$.
The inequality $B \leq B'$ is proven, and therefore $B = B' = \frac{\beta_1^2}{\beta_1^2 - 1}$.

\subsubsection{Computation of 
\texorpdfstring{$C$ and $D$}{C and D}}
If we consider $c_n := |\psi(\varphi^n(01))|$, then we can see that $c_n$ satisfies the same recurrence relation as $a_n$, i.e., $c_{n+1} = 2c_n - c_{n-1} + c_{n-2}$
with initial conditions $c_0 = 7, c_1 = 13$, and $c_2 = 25$.

This recurrence relation can be solved and we obtain
$$c_n = A_2\beta_1^n + B_2 \lambda_1^n + C_2 \lambda_2^n,$$
where
$$\beta_1 \doteq 1.75488,\quad \lambda_1 \doteq 0.12256 + 0.74486 i, \quad \lambda_2 = \overline{\lambda_1}$$
are the roots of the polynomial $t^3 - 2t^2 + t - 1$, and
\begin{align*}
    A_2 &= \frac{7|\lambda_1|^2 - 13\operatorname{Re}(\lambda_1) + 25}{|\beta_1-\lambda_1|^2} \doteq 8.0149\,;\\
    B_2 &= \frac{7\beta_1\lambda_2 - 13 (\beta_1+\lambda_2)+25}{(\beta_1 - \lambda_1)(\lambda_2-\lambda_1)} \doteq -0.5075 + 0.6315 i\,;\\
    % |b| \doteq 0.81
    C_2 &= \overline{B_2}.
\end{align*}

The proof of $C'=D'=\frac{\beta_1^2}{\beta_1^2 - 1}$  proceeds in a similar fashion as that for $B'$ (resp., $A'$).
To complete the proof, we need to show firstly $C \leq C'$, i.e., for all $n \geq 0$ that
\begin{align*}
    \frac{3+ A_2\sum_{k=0}^n\beta_1^{2k} + B_2 \sum_{k=0}^n \lambda_1^{2k} + C_2 \sum_{k=0}^n\lambda_2^{2k} }{A_2\beta_1^{2n} + B_2 \lambda_1^{2n} + C_2 \lambda_2^{2n}} 
    \quad &\leq^? \quad  \frac{\beta_1^2}{\beta_1^2 - 1}\\
    (\beta_1^2-1) \left(3 + 2\operatorname{Re}\left(B_2\frac{\lambda_1^{2n+2}-1}{\lambda_1^2 - 1}\right)\right)
    \quad &\leq^? \quad 2\beta_1^2 \operatorname{Re}(B_2\lambda_1^{2n}) + A_2\,.\\
\end{align*}
    Now, we have for all $n\geq 2$
    \begin{align*}
    (\beta_1^2-1) \left(3 + 2\operatorname{Re}\left(B_2\frac{\lambda_1^{2n+2}-1}{\lambda_1^2 - 1}\right)\right )
    &\leq (\beta_1^2-1)\left(3 + 2\operatorname{Re}\left(\frac{B_2}{1-\lambda_1^2}\right) + 2|B_2|\frac{|\lambda_1|^{6}}{|1-\lambda_1^2|} \right)\,;\\
    2\beta_1^2 \operatorname{Re}(B_2\lambda_1^{2n}) + A_2 &\geq A_2-2\beta_1^2|B_2||\lambda_1|^4.\\
    \intertext{Substituting the values proves that}
     (\beta_1^2-1)\left(3 + 2\operatorname{Re}\left(\frac{B_2}{1-\lambda_1^2}\right) + 2|B_2|\frac{|\lambda_1|^{6}}{|1-\lambda_1^2|} \right) &< A_2-2\beta_1^2|B_2||\lambda_1|^4.
    \end{align*}
    For $n = 0$, we get $\frac{3 + c_0}{c_0} = 1 + \frac{3}{7}< C'$.
    For $n = 1$, we get $\frac{3 + c_0 + c_2}{c_2} = 1 + \frac{3+7}{25} < C'$.
The inequality $C \leq C'$ is proven.   Hence $C = C' = \frac{\beta_1^2}{\beta_1^2 - 1}$.

Finally, we need to prove $D \leq D'$, i.e., for all $n \geq 1$, that
\begin{align*}
    \frac{4 + A_2\sum_{k=1}^n\beta_1^{2k-1} + B_2 \sum_{k=1}^n \lambda_1^{2k-1} + C_2 \sum_{k=1}^n\lambda_2^{2k-1} }{A_2\beta_1^{2n-1} + B_2 \lambda_1^{2n-1} + C_2 \lambda_2^{2n-1}} 
    \quad &\leq^? \quad  \frac{\beta_1^2}{\beta_1^2 - 1}\\
    (\beta_1^2-1) \left(4 + 2\operatorname{Re}\left(B_2\lambda_1\frac{\lambda_1^{2n}-1}{\lambda_1^2 - 1}\right)\right)
    \quad &\leq^? \quad 2\beta_1^2 \operatorname{Re}(B_2\lambda_1^{2n-1}) + A_2\beta_1\,.
\end{align*}

Using the same estimates as for $A$ and $n \geq 2$, it remains to check the following inequality
    \begin{align*}
    (\beta_1^2-1)\left(4 + 2 |B_2| |\lambda_1|\frac{|\lambda_1|^4+1}{|\lambda_1^2 - 1|}\right)
    &\leq A_2\beta_1 - 2\beta_1^2|B_2||\lambda_1|^{3}.
    \end{align*}
    The inequality holds for given values. 
    For $n = 1$, we get $\frac{4 + c_1}{c_1} = 1 + \frac{4}{13} < D'$. Consequently, $D = D'$.

We have shown that $A = B = C = D = \frac{\beta_1^2}{\beta_1^2-1}$ and $E = 1 + \frac{1}{3} < A$. In other words, we have proved the following theorem.
\begin{theorem}
    The critical exponent of $\psi({\bf p})$ equals 
    \begin{equation*}
        \cexp(\psi({\bf p})) = 1 + \frac{\beta_1^2}{\beta_1^2-1} = 2+ \frac{1}{\beta_1^2-1} \doteq 2.4808627161472369.        
    \end{equation*}
\end{theorem}
Let us emphasize that the critical exponents of ${\bf p}$ and $\psi(\bf{p})$ are the same~\cite{Currie&Ochem&Rampersad&Shallit:2022}.
\subsection{Bispecial factors in \texorpdfstring{$\xi(\bf{p})$}{xi(p)}}
\label{subsect:xi_bispecial factors}
Now we turn to $\xi(\bf{p})$.
Let us start with some simple observations. 
If a factor $v \in \LL(\xi({\bf p}))$ contains the factor 00 or 1011 or 11010, then we are able to express
$v = x\xi(w)y$ uniquely, where $w \in \LL({\bf p})$ and $x$ (resp., $y$) is a proper suffix (resp., proper prefix) of an image of some letter. 
Moreover, if $v$ is a bispecial factor in $\xi({\bf p})$, then $w$ is a bispecial factor in $\bf p$.

\begin{observation}\label{obs:BSxi}
    Let  $v \in \LL(\xi({\bf p}))$ be a bispecial factor such that it contains 00 or 1011 or 11010.
    Assume $v \not =1\xi(1)=10110$ and $v\not =1\xi(10)=1011001$.
    Then there exists $w\in \LL({\bf p})$ such that either $v =\xi(w)$ and $w$ is a bispecial factor with left extensions $\{0,1\}$, or $v=01\xi(w)$ and $w$ is a bispecial factor with left extensions $\{0,2\}$.
    Moreover, the length of the shortest return word to $v$ in $\xi({\bf p})$ equals $|\xi(r)|$, where $r$ is the shortest return word to $w$ in $\bf p$.
\end{observation}

\subsection{Critical exponent of 
\texorpdfstring{$\xi({\bf p})$}{xi(p)}}
Using Observation~\ref{obs:BSxi}, the lengths of bispecial factors and their shortest return words in $\xi({\bf p})$ may be derived in an analogous way as in the case of $\psi({\bf p})$. Next, we can use Theorem~\ref{Asy:computation} to compute the critical exponent of $\xi({\bf p})$:

\begin{align*}
\cexp(\xi({\bf p})) &= 1 + \max\left\{A,B,C,D,E\right\}\\
A &= \sup\left\{\frac{|01\xi(1\varphi(012)\varphi^3(012)\cdots \varphi^{2n-1}(012))|}{|\xi(\varphi^{2n-1}(012))|} : n\geq 1 \right\}\,; \\
B &= \sup\left\{\frac{|\xi(012\varphi^2(012)\varphi^4(012)\cdots \varphi^{2n}(012))|}{|\xi(\varphi^{2n}(012))|} : n\geq 0 \right\}\,; \\
C &= \sup\left\{\frac{|01\xi(01\varphi^2(01)\varphi^4(01)\cdots \varphi^{2n}(01))|}{|\xi(\varphi^{2n}(01))|} : n\geq 0 \right\}\,; \\
D &= \sup\left\{\frac{|\xi(\varphi(01)\varphi^3(01)\cdots \varphi^{2n+1}(01))|}{|\xi(\varphi^{2n+1}(01))|} : n\geq 0 \right\}\,; \\
E &=\max \left\{\frac{|w|}{|r|} : w \text{\ short bispecial and}\ r \text{ the shortest return word to } w \right\}\,,
\end{align*}
where we say that $w$ is a {\it short bispecial factor\/} if $w$ does not contain $00$, $1011$, $11010$ or $w=1\xi(1)$ or $w=1\xi(10)$.

\subsubsection{Computation of \texorpdfstring{$E$}{E}}
\begin{table}[h]
    \centering
        \setlength{\tabcolsep}{3pt}
\renewcommand{\arraystretch}{1.2}

\begin{tabular}{|c|c|c|c|c|c|c|c|c|c|c|}
    \hline
     $w$ & 0 & 1 & 01 & 10 & 101 & 0110 & 01101 &$1\xi(1)$& $1\xi(10)$\\ 
 \hline
     $r$ & 0 & 1 & 01 & 10 & 10 & 011 & 011010110&$1\xi(1)$& $1\xi(10)1^{-1}$\\
 \hline      
\end{tabular}
\end{table}

Thus, we have $E= \frac{3}{2}$. 

\subsubsection{Computation of \texorpdfstring{$A, B, C, D$}{A,B,C,D}}
If we consider $x_n := |\xi(\varphi^n(012))|$, then we can see that $x_n$ satisfies the recurrence relation 
$x_{n+1} = 2x_n - x_{n-1} + x_{n-2}$
with initial conditions $x_0 = 7, x_1 = 13$, and $x_2 = 24$.

This recurrence relation may be solved and we obtain
$$x_n = X_1\beta_1^n + Y_1 \lambda_1^n + Z_1 \lambda_2^n,$$
where
$$\beta_1 \doteq 1.75488,\quad \lambda_1 \doteq 0.12256 + 0.74486 i, \quad \lambda_2 = \overline{\lambda_1}$$
are the roots of the polynomial $t^3 - 2t^2 + t - 1$, and
\begin{align*}
    X_1 &= \frac{7|\lambda_1|^2 - 26\operatorname{Re}(\lambda_1) + 24}{|\beta_1-\lambda_1|^2} \doteq 7.704\,;\\
    Y_1 &= \frac{7\beta_1\lambda_2 - 13 (\beta_1+\lambda_2)+24}{(\beta_1 - \lambda_1)(\lambda_2-\lambda_1)} \doteq -0.352 + 0.291 i\,;\\
    % |b| \doteq 0.7039
    Z_1 &= \overline{Y_1}\,.
\end{align*}

We want to show for all $n\geq 1$ that
\begin{align*}
    \frac{6 + X_1\sum_{k=1}^n\beta_1^{2k-1} + Y_1 \sum_{k=1}^n \lambda_1^{2k-1} + Z_1 \sum_{k=1}^n\lambda_2^{2k-1} }{X_1\beta_1^{2n-1} + Y_1 \lambda_1^{2n-1} + Z_1 \lambda_2^{2n-1}} 
    \quad &\leq^? \quad \frac{\beta_1^2}{\beta_1^2-1}< \frac{3}{2}\,.\\
    \end{align*}
\begin{align*}
     (\beta_1^2-1)\left(6+ 2\operatorname{Re}\left(Y_1\lambda_1\frac{\lambda_1^{2n}-1}{\lambda_1^2 - 1}\right)\right) + X_1\beta_1^{2n+1} - X_1\beta_1
    \quad &\leq^? \quad 2\beta_1^2 \operatorname{Re}(Y_1\lambda_1^{2n-1}) + X_1\beta_1^{2n+1}\\
    (\beta_1^2-1) \left(6 + 2\operatorname{Re}\left(Y_1\lambda_1\frac{\lambda_1^{2n}-1}{\lambda_1^2 - 1}\right)\right)
    \quad &\leq^? \quad 2\beta_1^2 \operatorname{Re}(Y_1\lambda_1^{2n-1}) + X_1\beta_1\,.\\
\end{align*}
Now, on one hand,  for $n\geq 7$ we have
\begin{align*}
    (\beta_1^2-1) \left(6 + 2\operatorname{Re}\left(Y_1\lambda_1\frac{\lambda_1^{2n}-1}{\lambda_1^2 - 1}\right)\right)  
    &\leq (\beta_1^2-1)\left(6 + 2 |Y_1| |\lambda_1|\frac{|\lambda_1|^{14}+1}{|\lambda_1^2-1| }\right)\,.
    \end{align*}
On the other hand,  we have   
\begin{align*}
    2\beta_1^2 \operatorname{Re}(Y_1\lambda_1^{2n-1}) + X_1\beta_1
    &\geq X_1\beta_1 - 2\beta_1^2|Y_1||\lambda_1|^{13}.
    \intertext{And if we substitute the values, we get}
    (\beta_1^2-1)\left(6 + 2 |Y_1| |\lambda_1|\frac{|\lambda_1|^{14}+1}{|\lambda_1^2-1| }\right)
    &\leq X_1\beta_1 - 2\beta_1^2|Y_1||\lambda_1|^{13}\,.
    \end{align*}

    The first 6 values from the set $A$ are given in Table~\ref{tab:elementsofA}. 
    We can see that they are all smaller than \mbox{$\frac{3}{2}$}. The inequality therefore holds for all $n\geq 1$.
    \begin{table}[htb]
        \centering
        \begin{tabular}{c|cccccc}
            $n$ & 1 & 2 & 3 & 4 & 5 & 6 \\\hline
            $A_n$ & 1.4615 & 1.4524 & 1.4766 & 1.4785 & 1.4803 & 1.4806 
        \end{tabular}
        \caption{The first 6 elements of $A$.}
        \label{tab:elementsofA}
    \end{table}
    
We have shown that $A<\frac{3}{2}=E$. Using the same estimates, we can show that also the remaining values $B,C,D$ are smaller than $\frac{3}{2}$. In other words, we have proved the following theorem.
\begin{theorem}
    The critical exponent of $\xi({\bf p})$ equals $\frac{5}{2}$.
\end{theorem}

\section{Complexity threshold}
\label{threshsec}

Notice that we used two different constructions of infinite words in the proof of Theorem~\ref{complem1}.
When possible, that is in cases $(b)$, $(c)$, and $(f)$, we use morphic images of arbitrary
ternary square-free words, thus showing that exponentially many binary words have the considered property.
Otherwise, in cases $(a)$, $(d)$, and $(e)$, every bi-infinite binary word with the considered property has the same set of factors as one or a few morphic words.

In this section, we consider the latter case and we relax the constraints on
the critical exponent or the complementary factors in order to get exponential factor complexity.

Let $h$ be the morphism that maps $0 \mapsto 0$ and $1 \mapsto 01$.

\begin{lemma}\label{lemhfib}
 Every bi-infinite binary $4$-free word avoiding pairs of complementary factors of length 4 and $\{1001001,0110110\}$
 has the same set of factors as either $h({\bf f})$ or $\overline{h({\bf f})}$.
\end{lemma}

\begin{proof}
We compute the set $S$ of factors $y$ such that there exists a binary $4$-free word $xyz$ avoiding pairs of complementary factors of length 4 and $\{1001001,0110110\}$
and such that $|x|=|y|=|z|=100$.
We observe that $S$ consists of the (disjoint) union of the factors of length 100 of $h({\bf f})$ and $\overline{h({\bf f})}$.
Let $\bf w$ be a bi-infinite binary $4$-free word avoiding pairs of complementary factors of length 4 and $\{1001001,0110110\}$.
By symmetry, we suppose that $\bf w$ contains $00$, so that $\bf w$ has the same factors of length 100 as $h({\bf f})$.
So ${\bf w}\in\{0,01\}^\omega$, that is, ${\bf w}=h({\bf v})$ for some bi-infinite binary word $\bf v$.
Since $\bf w$ is $4$-free, so is $\bf v$.
Moreover, by considering the pre-images by $h$ of ${\bf w}=h({\bf v})$ and $h({\bf f})$,
this implies that $\bf v$ and $\bf f$ have the same set of factors of length $100/\max(|h(0)|,|h(1)|)=50$.
In particular, $\bf v$ avoids $\{11,000,10101\}$.
By \cite[Thm.~8]{Baranwal&Currie&Mol&Ochem&Rampersad&Shallit:2022},
$\bf v$ has the same set of factors as $\bf f$.
Thus $\bf w$ has the same set of factors as $h({\bf f})$.
\end{proof}

\begin{theorem}\label{l=4}
Every bi-infinite binary $\tfrac{11}3$-free word that contains no pair
of complementary factors of length $4$ has the same set of factors as  $h({\bf f})$ or the same set of factors as $\overline{h({\bf f})}$.
\end{theorem}

\begin{proof}
% Suppose that $\bf w$ is a bi-infinite binary word that contains no pair
% of complementary factors of length $4$ and has critical exponent smaller than $(5+\sqrt{5})/2$.
Suppose that such a word $\bf w$ contains $1001001$.

First, suppose that $\bf w$ contains $11$.
Since $\bf w$ avoids $0110=\overline{1001}$, $\bf w$ contains $111$.
Since $\bf w$ avoids the $4$-power $1111$, $\bf w$ contains $01110$.
Since $\bf w$ avoids $1101=\overline{0010}$, $\bf w$ contains $011100$.
Symmetrically, since $\bf w$ avoids $1011=\overline{0100}$, $\bf w$ contains $0011100$.
So $\bf w$ contains both $0011$ and $1100$, a contradiction.
Thus, $\bf w$ avoids $11$.

Now, suppose that $\bf w$ contains $101$.
Since $\bf w$ avoids $11$, $\bf w$ contains $01010$.
So $\bf w$ contains both $0101$ and $1010$, a contradiction.
Thus, $\bf w$ avoids $101$.

Since $\bf w$ contains $1001001$ and avoids $11$, $\bf w$ contains $010010010$.
Since $\bf w$ contains $010010010$ and avoids $101$, $\bf w$ contains $00100100100$.
This is a contradiction, since $00100100100=(001)^{11/3}$.

So $\bf w$ avoids $1001001$.
By symmetry, $\bf w$ avoids $0110110$.

By Lemma~\ref{lemhfib}, $\bf w$ has the same set of factors as either $h({\bf f})$ or $\overline{h({\bf f})}$.
%So ${\bf w}$ has critical exponent $(5+\sqrt{5})/2$.
\end{proof}

By contrast, there exist exponentially many binary $\tfrac{11}3^+$-free words
with no pair of complementary factors of length $4$.

\begin{theorem}
% The image of any ternary squarefree word by the 21-uniform morphism
% \begin{align*}
% 0 &\rightarrow 000100100100010001001 \\
% 1 &\rightarrow 000100010010010001001 \\
% 2 &\rightarrow 000100010010001001001
% \end{align*}
The image of any ternary squarefree word by the 31-uniform morphism
\begin{align*}
0 &\rightarrow 0010001001001000100100100010010 \\
1 &\rightarrow 0010001001001000100100100010001 \\
2 &\rightarrow 0010001001001000100100010010010
\end{align*}
is a $\tfrac{11}3^+$-free word containing only $0, 1, 01$, and 
 $10$ as complementary factors.
\label{29o11p}
\end{theorem}

Let $\rho$ be the morphism defined as follows
\begin{align*}
    0 &\rightarrow 01100101101\\
    1 &\rightarrow 0110010\\
    2 &\rightarrow 011001.
\end{align*}
%to the infinite word $\bf p$ of Section~\ref{notation}.

Using the technique of Theorem~\ref{complem1} (d), we can show that $\rho({\bf p})$ is a
$\tfrac52^+$-free word with no pair of complementary factors of length 8 and exactly 40 complementary factors.

\begin{theorem}\label{l=8}
Every bi-infinite binary $\tfrac{29}{11}$-free word that contains no pair
of complementary factors of length $8$ has the same set of factors as either
$\xi({\bf p})$, $\overline{\xi({\bf p})}$, $\xi({\bf p})^R$, $\overline{\xi({\bf p})^R}$,
$\rho({\bf p})$, $\overline{\rho({\bf p})}$, $\rho({\bf p})^R$, or $\overline{\rho({\bf p})^R}$.
\end{theorem}

\begin{proof}
First, we show the following. If $\bf w$ is a bi-infinite cube-free binary word and every
factor of length $21$ of $\bf w$ is also a factor of $\xi({\bf p})$, then $\bf w$ has 
the same set of factors as $\xi({\bf p})$.

We compute the set $S$ of factors $f$ such that there exists $u=efg$,
where $u$ is cube-free, every factor of length $21$ of $u$
is a factor of $\xi({\bf p})$, and $|e| = |f| = |g| = 500$.
We verify that every factor $f\in S$ is a factor of $\xi({\bf p})$.

This means that ${\bf w}=\xi({\bf v})$ for some bi-infinite ternary word $\bf v$.
Moreover, by considering the pre-images by $\xi$ of ${\bf w}=\xi({\bf v})$
and $\xi({\bf p})$, this implies that $\bf v$ and $\bf p$ have the same set
of factors of length $500/\max(|\xi(0)|,|\xi(1)|,|\xi(2)|)=125$.
In particular, $\bf v$ has the same set of factors as $\bf p$ by \cite[Theorem 14]{Currie&Ochem&Rampersad&Shallit:2022}.

Similarly, we show that if $\bf w$ is a bi-infinite cube-free binary word and every
factor of length $63$ of $\bf w$ is also a factor of $\rho({\bf p})$, then $\bf w$ has 
the same set of factors as $\rho({\bf p})$.

Finally, we compute the set $X$ of factors $f$ such that there exists $u=efg$,
where $u$ is $\tfrac{29}{11}$-free, $u$ contains no pair
of complementary factors of length $8$, $|f| = 200$, and $|e| = |g| = 80$.
We verify that every factor $f\in X$ is a factor of either
$\xi({\bf p})$, $\overline{\xi({\bf p})}$, $\xi({\bf p})^R$, $\overline{\xi({\bf p})^R}$,
$\rho({\bf p})$, $\overline{\rho({\bf p})}$, $\rho({\bf p})^R$, or $\overline{\rho({\bf p})^R}$. Since $\tfrac{29}{11}$-free words are cube-free,
Theorem~\ref{l=8} follows by the two previous results and the symmetries by complement and reversal.
\end{proof}

By contrast, there exist exponentially many binary $\tfrac{29}{11}^+$-free words containing no pair
of complementary factors of length $8$ and exactly 36 complementary factors.

\begin{theorem}
% The image of any ternary squarefree word by the 42-uniform morphism
% \begin{align*}
% 0 &\rightarrow 001010010110010011001011001001100101001011 \\
% 1 &\rightarrow 001010010110010011001010010110010011001011 \\
% 2 &\rightarrow 001001100101001011001010010110010011001011
% \end{align*}
The image of any ternary squarefree word by the 84-uniform morphism
{\footnotesize
\begin{align*}
0 &\rightarrow 100101100100110010100101100101001011001001100101100101001011001001100101100100110010 \\
1 &\rightarrow 100101100100110010100101100101001011001001100101100101001011001001100101001011001001 \\
2 &\rightarrow 100101100100110010100101100101001011001001100101100100110010100101100100110010110010
\end{align*}
}
is a $\tfrac{29}{11}^+$-free word with no pair of complementary factors of length $8$ and exactly 36 complementary factors.
\label{29o11p2}
\end{theorem}

We also get exponentially many words if we allow complementary factors of length~8.

\begin{theorem}
The image of any ternary squarefree word by the 31-uniform morphism
\begin{align*}
0 &\rightarrow 0010100110010011001010011001011 \\
1 &\rightarrow 0010100101100101001100101001011 \\
2 &\rightarrow 0010011001011001010010110010011
\end{align*}
is a $\tfrac{5}{2}^+$-free word with no pair of complementary factors of length $9$ and exactly 40 complementary factors.
\label{l=9}
\end{theorem}

\end{document}